\documentclass[11pt]{article}

\usepackage[english]{babel}
\usepackage{amsthm}
\newtheorem*{theorem}{Theorem}
\newtheorem{lemma}{Lemma}

\usepackage{float}

\usepackage{amsmath, amsfonts, amssymb}
\usepackage{graphicx}
\usepackage{epstopdf}
\usepackage{bm}
\usepackage{color}
\usepackage{lscape}
\usepackage{url}
\usepackage{hyperref}
\usepackage{multirow}

\usepackage{algorithm}
\usepackage{algorithmicx}

\newcommand{\bd}[1]{\mathbf{#1}}

\newcommand{\del}{\delta}

\newcommand{\pa}{\partial}
\newcommand{\beq}{\begin{equation}}
\newcommand{\eeq}{\end{equation}}

\DeclareMathOperator\erf{erf}
\DeclareMathOperator\erfc{erfc}

\definecolor{red}{rgb}{1.0,0.0,0.0}

\title{The adjoint double layer potential on smooth surfaces in $\mathbb{R}^3$ and the Neumann problem}
\author{J. Thomas Beale\thanks{Department of Mathematics, Duke University, Durham, NC, 27708 USA beale@math.duke.edu} \and Michael Storm\thanks{Department of Mathematics, Farmingdale State College, SUNY, Farmingdale, NY 11735, USA storm@farmingdale.edu} \and Svetlana Tlupova\thanks{Department of Mathematics, Farmingdale State College, SUNY, Farmingdale, NY 11735, USA tlupovs@farmingdale.edu}}

\date{\today}

\begin{document}

\maketitle

\begin{abstract} 
We present a simple yet accurate method to compute the adjoint double layer potential, which is used to solve the Neumann boundary value problem for Laplace's equation in three dimensions.  An expansion in curvilinear coordinates leads us to modify the expression for the adjoint double layer so that the singularity is reduced when evaluating the integral on the surface.  We then regularize the Green's function, with a radial parameter $\delta$. We show that a natural regularization has error $O(\delta^3)$, and a simple modification improves the error to $O(\delta^5)$.  The integral is evaluated numerically without the need of special coordinates. We use this treatment of the adjoint double layer to solve the classical integral equation for the interior Neumann problem and evaluate the solution on the boundary. Choosing $\del = ch^{4/5}$, we find 
about $O(h^4)$ convergence in our examples, where $h$ is the spacing in a background grid.
\end{abstract}


 
%%%%%%%%%%%%%%%%%%%%%%%%%%%%%%%%%%%%%%%%%
%%%%%%%%%%%%%%%%%%%%%%%%%%%%%%%%%%%%%%%%%

\section{Introduction}
We consider the Neumann boundary value problem for Laplace's equation, 
\begin{equation}\label{naminn}
    \Delta u=0 \text{ on } \Omega, \quad \frac{\partial u}{\partial n} = g \text{ on } S,
\end{equation}
which finds applications in a wide range of scientific and engineering problems. Here $\Omega$ is a bounded domain in $\mathbb{R}^3$ with a smooth boundary $S$, and $g$ is a given function on $S$
with integral zero. A common approach to solve the Neumann problem~\eqref{naminn} is to use layer potentials, reducing the problem to solving an integral equation on the surface~\cite{kress-99, folland-95}. The solution $u$ is represented by a single layer potential,
\begin{equation}\label{sl}
    u(\textbf{y}) = \int_{S} G(\textbf{x} - \textbf{y}) \, f(\textbf{x}) \, dS(\textbf{x}),
\end{equation}
where $G$ is the Green's function, ${G(\textbf{x})=\frac{-1}{4\pi |\textbf{x}|}}$, and $f$ is a density function determined by solving the integral equation of the second kind,
\begin{equation}\label{huu}
    -\frac{1}{2} f(\textbf{y}) +\int_{S} \frac{\partial }{\partial n( \textbf{y})}G( \textbf{x}- \textbf{y})\, f(\textbf{x}) \, dS( \textbf{x}) = g(\textbf{y}),\quad  \textbf{y}\in S.
\end{equation}
The integral in~\eqref{huu} is called the adjoint double layer potential,
\begin{equation}\label{adjdl}
   v(\textbf{y})= \int_{S} \frac{\partial }{\partial n( \textbf{y})}G( \textbf{x}- \textbf{y})f(\textbf{x})\, dS( \textbf{x}),
\end{equation}
where $\textbf{n}$ is the outward unit normal at $\textbf{y}$ on the surface. The normal derivative of the Green's function is
\begin{equation}\label{4}
    \frac{\partial }{\partial n( \textbf{y})}G( \textbf{x}- \textbf{y})=\textbf{n}(\textbf{y})\cdot \nabla_{\textbf{y}} G(\textbf{x}-\textbf{y})=\frac{\textbf{n}(\textbf{y})\cdot (\textbf{y}- \textbf{x})}{4\pi |\textbf{x}-\textbf{y}|^3}.
\end{equation}
Thus, the integrand in~\eqref{adjdl} is singular when ${\textbf{y} = \textbf{x}}$. Once the integral equation~\eqref{huu} is solved for the density function, the solution~\eqref{sl} can be evaluated at points on the surface, where the integrand is again singular, or at points off the surface. When the evaluation point is very near the surface, the accurate numerical integration of the near singularity is also a challenge since the integrand will have large derivatives; e.g., see \cite{beale-ying-wilson-16}. Our aim in this paper is to analyze the integrand of the adjoint double layer~\eqref{adjdl} and modify it to reduce the singularity, then solve the integral equation in~\eqref{huu} with high accuracy.

The area of numerical integration of singular and nearly singular integrals is extensive, and the most widely used numerical technique is the boundary element method~\cite{atkinson-97, kress-99, hsiao-wendland-08, sauter-schwab-10}. Some approaches are based on singularity subtraction methods~\cite{helsing-13,perez}, direct (Nystr\"{o}m) quadrature based on partitions of unity and coordinate transformations~\cite{bruno-kunyansky-01, ying-biros-zorin-06}, auxiliary nodes on adaptive triangular discretizations~\cite{bremer-gimbutas-12}, quadrature by expansion~\cite{klockner-barnett-greengard-oneil-13, siegel-tornberg-18}, and regularization of the kernel~\cite{beale-04, beale-ying-wilson-16}. When the domain has corners, the solution may exhibit singularities, which presents additional difficulties, see~\cite{hoskins-rachh-20} for some recent work on the Neumann problem on polygonal domains. 

Singularity subtraction methods, such as in Helsing~\cite{helsing-13}, are based on splitting the integral operator into the singular part, which is evaluated analytically using product integration, and the regular part where a straightforward numerical quadrature can be used. Bruno and Kunyansky~\cite{bruno-kunyansky-01} developed a high-order integrator for surface scattering problems based on partitions of unity to split the integral into the adjacent (near) interactions, where the singularity is resolved analytically by a change to polar coordinates, and the remainder where FFT's are used. This method was extended in Ying, Biros and Zorin~\cite{ying-biros-zorin-06} for 3D elliptic problems on multiply-connected domains. In Bremer and Gimbutas~\cite{bremer-gimbutas-12}, weakly singular integrals are evaluated efficiently using triangular discretizations.  A quadrature by expansion (QBX) method was developed by Kl\"{o}ckner et al.~\cite{klockner-barnett-greengard-oneil-13} for evaluation of Laplace and Helmholtz potentials, through local expansions centered at points very close to the surface. This method can achieve exponential accuracy but requires upsampling the density on a finer grid. A target-specific QBX method was developed by Siegel and Tornberg~\cite{siegel-tornberg-18}, where the same accuracy is achieved using fewer terms.  In recent work~\cite{wu-mart} special weights are used on a regular grid near the singularity for toroidal surfaces, and in~\cite{izzorunborgtsai} the integrand is extended to points beyond the surface.

This paper presents a simple method for computing the adjoint double layer potential~\eqref{adjdl} accurately in the singular case, without requiring  special quadrature or detailed surface representations. First, we apply a singularity modification that results in an integrand in~\eqref{adjdl} that is bounded but not smooth, making it more amenable to straightforward numerical integration. This modification is similar to the singularity subtraction that is very well known for the double layer potential but to the best of our knowledge, has only been briefly noted for the adjoint in~\cite{zinchenko-rother-davis-97} and~\cite{pozrikidis-01}. We 
demonstrate the effect of the singularity modification using Taylor expansions of the position and the normal vector in the curvilinear coordinates of the surface. With the bounded integrand, the kernel is then regularized following the method developed by Beale and collaborators~\cite{beale-04, beale-ying-wilson-16}. We use a special regularization for high accuracy, so that the resulting error is $O(\del^5)$, where $\del$ is the regularization parameter. We demonstrate this high accuracy using several numerical examples to evaluate the integral in~\eqref{adjdl}, solve the integral equation~\eqref{huu}, and
then find the solution of the Neumann problem on the surface \eqref{sl} using the method of
\cite{beale-ying-wilson-16}. 
We use a simple direct quadrature developed in~\cite{beale-ying-wilson-16, wilson-10} which does not require coordinate charts or triangulations on the surface, and is high order accurate for smooth surfaces and integrands.  In the present case, with mesh spacing $h$, we choose $\delta = ch^q$ with $q<1$ and find accuracy about $O(h^4)$ for $q = 4/5$.  This approach has been shown to be effective for single and double layer integrals in the context of Laplace and Stokes equations~\cite{beale-ying-wilson-16, tlupova-beale-19}. 

The rest of the paper is organized as follows. In Section~\ref{sec: Preliminaries}, we review the Taylor expansions for the analysis of the kernel singularities and the reduction of the singularity in the adjoint double layer in Section~\ref{sec: Sing_mod}. In Section~\ref{sec: Regularization}, we describe the method of regularization, estimate the resulting error, and present the improvement with higher accuracy.  We briefly discuss the quadrature rule and then the discrete form of the integral equation. In Section~\ref{sec: Numerical Results}, we illustrate the method with several numerical examples.

%%%%%%%%%%%%%%%%%%%%%%%%%%%%%%%%%%%%%%%%%
%%%%%%%%%%%%%%%%%%%%%%%%%%%%%%%%%%%%%%%%%

\section{Preliminaries}
\label{sec: Preliminaries}

In this section, we discuss the surface parametrization and
Taylor expansions that will be used in our analysis of the kernel singularity. For simplicity, we assume ${\textbf{y}=\textbf{0}}$, and use the parametric representation 
\begin{equation}
     \textbf{x}(\alpha)=\langle \, x_1(\alpha_1,\alpha_2),x_2(\alpha_1,\alpha_2),x_3(\alpha_1,\alpha_2)\, \rangle
\end{equation}
to describe the surface $S$ nearby.
We assume that ${\textbf{x}(0)=\textbf{0}}$. Let ${\textbf{T}_1=\frac{\partial  \textbf{x}}{\partial \alpha_1}}$ and ${\textbf{T}_2=\frac{\partial  \textbf{x}}{\partial \alpha_2}}$ be the tangent vectors to the surface; the metric tensor is ${g_{ij}=\textbf{T}_i\cdot \textbf{T}_j}$, ${i,j=1,2}$ and ${\textbf{n}_0= \textbf{n}(0)}$ is the unit normal at ${\alpha =0}$. We make some simplifying assumptions about the $\alpha$-coordinate system; in particular, we assume the metric tensor is the identity at ${\alpha =0}$, ${g_{ij}=\del_{ij}}$. We also assume that ${\frac{\partial g_{ij}}{\partial \alpha_k}=0}$, ${i,j,k=1,2}$ at $\alpha = 0$, which is equivalent to the Christoffel symbols being zero at ${\alpha =0}$. Furthermore, we assume that ${\textbf{T}_1, \textbf{T}_2}$ have the directions of principal curvature at ${\alpha =0}$.
Since ${\frac{\partial g_{ij}}{\partial \alpha_k}=0}$, it can be shown that ${\textbf{T}_i \cdot \frac{\partial \textbf{T}_j}{\partial \alpha_k}=0}$ at $\alpha = 0$.

We now look at the Taylor expansions for ${\textbf{x}(\alpha )}$ near 0,
\begin{equation}\label{xtaylor}
    \textbf{x}(\alpha )=\textbf{T}_1 \alpha _1 +\textbf{T}_2 \alpha _2 + \frac{1}{2} \frac{\partial \textbf{T}_1}{\partial \alpha _1}\alpha _1 ^2+ \frac{\partial \textbf{T}_2}{\partial \alpha _1}\alpha _1 \alpha _2 +\frac{1}{2} \frac{\partial \textbf{T}_2}{\partial \alpha _2}\alpha _2 ^2 +O(|\alpha |^3),
\end{equation}
where the tangents and their derivatives are evaluated at ${\alpha =0}$. {Since ${\{ \textbf{T}_1, \textbf{T}_2, \textbf{n}_0 \}}$ form an orthonormal basis at ${\alpha =0}$, we can write}
\begin{equation}\label{pen}
    \frac{\partial \textbf{T}_1}{\partial \alpha _1}=(\frac{\partial \textbf{T}_1}{\partial \alpha _1}\cdot \textbf{T}_1)\textbf{T}_1+(\frac{\partial \textbf{T}_1}{\partial \alpha _1}\cdot \textbf{T}_2)\textbf{T}_2+(\frac{\partial \textbf{T}_1}{\partial \alpha _1}\cdot \textbf{n}_0 )\textbf{n}_0 ,
\end{equation}
and since ${\textbf{T}_i \cdot \frac{\partial \textbf{T}_j}{\partial \alpha_k}=0}$, the first two terms are 0. Thus when ${\alpha=0}$, equation~\eqref{pen} becomes
\begin{equation}
   \frac{\partial \textbf{T}_1}{\partial \alpha _1}=\kappa _1 \textbf{n}_0  .
\end{equation}
Similarly, ${\frac{\partial \textbf{T}_2}{\partial \alpha _2}=\kappa _2 \textbf{n}_0}$, where ${\kappa _1 ,\kappa _2}$ are the principal curvatures. As for $\frac{\partial \textbf{T}_1}{\partial \alpha _2}$ which is equal to $\frac{\partial \textbf{T} _2}{\partial \alpha _1}$ for a smooth surface, we have
\begin{equation}
    \frac{\partial \textbf{T}_1}{\partial \alpha _2} = (\frac{\partial \textbf{T}_1}{\partial \alpha _2} \cdot  \textbf{T}_1)\textbf{T}_1+(\frac{\partial \textbf{T}_1}{\partial \alpha _2} \cdot \textbf{T}_2)\textbf{T}_2+(\frac{\partial \textbf{T}_1}{\partial \alpha _2} \cdot \textbf{n}_0)\textbf{n}_0 = (\frac{\partial \textbf{T}_1}{\partial \alpha _2} \cdot \textbf{n}_0 )\textbf{n}_0 ,
\end{equation}
and this is 0 as well since $\textbf{T}_1$ and $\textbf{T}_2$ have the directions of principal curvature. Then~\eqref{xtaylor} becomes 
\begin{equation} \label{xtaylor2}
    \textbf{x}(\alpha )= \textbf{T}_1 \alpha_1 +\textbf{T}_2 \alpha_2 +\frac{1}{2} \kappa _1 \textbf{n}_0 \alpha _1 ^2 +\frac{1}{2} \kappa _2 \textbf{n}_0 \alpha _2 ^2 +O(|\alpha |^3),
\end{equation}
see equation (2.6) in~\cite{beale-04}. 

For the normal vector ${\textbf{n}(\alpha)}$, we have
\begin{equation}\label{ntaylor}
    \textbf{n}(\alpha ) = \textbf{n}_0 + \frac{\partial \textbf{n}}{\partial \alpha _1}(0) \alpha _1 + \frac{\partial \textbf{n}}{\partial \alpha _2}(0) \alpha _2 + O(|\alpha |^2),
\end{equation}
where
\begin{equation}\label{2nn}
    \textbf{n} = \frac{\textbf{T}_1 \times \textbf{T}_2 }{\sqrt{g}},
\end{equation}
with
\begin{equation}\label{apfel}
    g=g_{11}g_{22}-g_{12}g_{21}=|\textbf{T}_1 \times \textbf{T}_2 |^2.
\end{equation} 
To find $\frac{\partial \textbf{n}}{\partial \alpha _1}$ we apply $\pa/\pa \alpha_1$ to the identity
$\textbf{n}(\alpha)\cdot\textbf{T}_1(\alpha) \equiv 0$, finding at $\alpha = 0$,
\beq\label{dnt1}
\frac{\partial \textbf{n}}{\partial \alpha _1}(0)\cdot\textbf{T}_1 (0) = 
  - \textbf{n}_0 \cdot\frac{\partial \textbf{T}_1}{\partial \alpha _1}(0) = -\kappa_1
\eeq 
In the same way, from $\textbf{n}\cdot\textbf{T}_2 \equiv 0$ and
$\textbf{n}\cdot\textbf{n} \equiv 1$, we get at $\alpha = 0$
\beq \frac{\partial \textbf{n}}{\partial \alpha _1}\cdot\textbf{T}_2 = 0\,,\quad
\frac{\partial \textbf{n}}{\partial \alpha _1}\cdot \textbf{n}_0 = 0\,.
\eeq
Thus
\begin{equation}
    \frac{\partial \textbf{n}}{\partial \alpha _1}(0) = -\kappa _1 \textbf{T}_1 (0).
\end{equation}
Similarly we can find that
\begin{equation}
    \frac{\partial \textbf{n}}{\partial \alpha _2}(0) = -\kappa _2 \textbf{T}_2 (0).
\end{equation}
Then the expansion in~\eqref{ntaylor} becomes
\begin{equation}\label{ntaylor2}
    \textbf{n}(\alpha ) = \textbf{n}_0-\kappa _1 \textbf{T}_1 (0) \alpha _1 -\kappa _2 \textbf{T}_2 (0) \alpha _2 +O(|\alpha |^2),
\end{equation}
see (2.7) in~\cite{beale-04}. Also since ${\textbf{y} = \textbf{0}}$,
\begin{equation}
    r=\sqrt{|\textbf{x} - \textbf{y}|^2} = \sqrt{\alpha _1 ^2 + \alpha _2 ^2 + O(|\alpha |^4 )} = | \alpha |  +O(|\alpha |^3 ).
\end{equation}
We also have the expansion for $f$,
\begin{equation}\label{ftaylor}
    f(\alpha ) = f_0 +f_1 \alpha _1 + f_2 \alpha _2 + O(|\alpha |^2 ),
\end{equation}
where ${f_0=f(0)}$, ${f_1=\frac{\partial f}{\partial \alpha _1}(0)}$, ${f_2=\frac{\partial f}{\partial \alpha _2}(0)}$.\\

%%%%%%%%%%%%%%%%%%%%%%%%%%%%%%%%%%%%%%%%%
%%%%%%%%%%%%%%%%%%%%%%%%%%%%%%%%%%%%%%%%%

\section{Singularity reduction}
\label{sec: Sing_mod}

We will use the expansions from the previous section to modify the singularity in the adjoint double layer so that the new integrand is bounded. We start with the more familiar double layer potential.
\begin{lemma}
The double layer potential 
\begin{equation}\label{DoubleLayer}
    \int_{S} \frac{\textbf{n}(\textbf{x})\cdot (\textbf{x}-\textbf{y})}{4\pi |\textbf{x}-\textbf{y}|^3}f(\textbf{x}) \, dS( \textbf{x})  ,
\end{equation}
has a singular integrand for ${\textbf{x}=\textbf{y}}$.
\end{lemma}
\begin{proof} After the variable change to $\alpha$ and Taylor expansions in~\eqref{xtaylor2} and \eqref{ntaylor2}, we get
\begin{equation}
     \textbf{n}(\textbf{x}) \cdot (\textbf{x}-\textbf{y}) = \textbf{n}(\alpha ) \cdot \textbf{x}(\alpha )  = -\frac{1}{2} (\kappa _1 \alpha _1 ^2 + \kappa _2 \alpha _2 ^2 )+O(|\alpha |^3 ).
\end{equation}
Thus the integrand has the form
\begin{equation}\label{a}
   \frac{( \textbf{n}(\textbf{x})\cdot (\textbf{x}-\textbf{y}))f(\textbf{x})}{4 \pi |\textbf{x} - \textbf{y}|^3} = \frac{ -\frac{1}{2} f_0 (\kappa _1 \alpha_1 ^2 +\kappa _2 \alpha _2 ^2 ) +O(|\alpha |^3 )}{4\pi (|\alpha |+O(|\alpha |^3 ))^3},
\end{equation}
which is singular at ${\alpha =0}$.
\end{proof}

In order to reduce the singularity in the double layer, the well-known identity
\begin{equation}\label{sub}
     \int_{S} \frac{\textbf{n}(\textbf{x})\cdot (\textbf{x}-\textbf{y})}{4\pi |\textbf{x}-\textbf{y}|^3} \, dS(\textbf{x}) = \frac{1}{2},\quad \textbf{y}\in S,
\end{equation}
is often used.
\begin{lemma}
The double layer potential \eqref{DoubleLayer} can be expressed as
\begin{equation}\label{dlsub}
    \int_{S} \frac{\textbf{n}(\textbf{x})\cdot (\textbf{x} - \textbf{y})}{4\pi |\textbf{x} - \textbf{y}|^3}\, (f(\textbf{x}) - f(\textbf{y}))\, dS( \textbf{x}) + \frac{1}{2} f(\textbf{y}),
\end{equation}
so that the integrand is bounded at ${\textbf{x}=\textbf{y}}$.
\end{lemma}
\begin{proof}
Using the expansion in~\eqref{ftaylor}, we can write
\begin{equation}
    f(\textbf{x})-f(\textbf{y}) = f(\alpha )-f(0) = f_1 \alpha _1 +f_2 \alpha _2 +O(|\alpha |^2),
\end{equation}
and the integrand in~\eqref{dlsub} becomes
\begin{equation}\label{c}
    \frac{(\textbf{n}(\alpha ) \cdot \textbf{x}(\alpha )) \, (f(\alpha )-f(0))}{4\pi |\textbf{x}(\alpha ) |^3} = \frac{-\frac{1}{2}(\kappa _1 \alpha _1 ^2 +\kappa _2 \alpha _2 ^2 )\, (f_1 \alpha _1 +f_2 \alpha _2 ) + O(|\alpha |^4 )}{4\pi (|\alpha | + O(|\alpha |^3 ))^3},
\end{equation}
which is bounded.
\end{proof}
We now turn our attention to the adjoint double layer integral~\eqref{adjdl}. We first show the integrand is singular, and then propose a modification that gives a bounded integrand.
\begin{lemma}
The adjoint double layer potential,
\begin{equation}\label{adl}
 v(\bd{y}) = \int_{S}   \frac{\textbf{n}(\textbf{y}) \cdot (\textbf{y}-\textbf{x})}{4\pi |\textbf{x} - \textbf{y}|^3} f(\textbf{x})\, dS(\textbf{x}),
\end{equation}
has a singular integrand at ${\textbf{x}=\textbf{y}}$.
\end{lemma} 
\begin{proof}
Using the expression in~\eqref{xtaylor2} and the assumption ${\textbf{y} = 0}$,
\begin{equation}
    \textbf{n}(\textbf{y})\cdot (\textbf{y} -\textbf{x}) = -\textbf{n}_0 \cdot \textbf{x}(\alpha ) = -\frac{1}{2}(\kappa _1 \alpha _1 ^2 + \kappa _2 \alpha _2 ^2) +O(|\alpha |^3 ),
\end{equation}
and the integrand in~\eqref{adl} becomes
\begin{equation}\label{b}
    \frac{ - \frac{1}{2}f_0 (\kappa _1 \alpha _1 ^2 +\kappa _2 \alpha _2 ^2)+O(|\alpha |^3)}{4\pi (|\alpha |+O(|\alpha |^3 ))^3},
\end{equation}
which is singular at ${\alpha = 0}$, similarly to~\eqref{a}.
\end{proof}
We cannot use subtraction in the adjoint double layer to reduce the singularity in the same way it was done in~\eqref{dlsub} for the double layer potential. Nonetheless, the expansions in~\eqref{b} and~\eqref{a} suggest the following modification to the adjoint double layer. If ${\textbf{n}(\textbf{y})}$ is replaced with ${\textbf{n}(\textbf{y})+\textbf{n}(\textbf{x})}$ in~\eqref{adl}, the integrand is bounded at ${\textbf{x}=\textbf{y}}$ due to the leading terms in~\eqref{b} and~\eqref{a} cancelling exactly and the remainder being ${O(|\alpha |^3 )}$; see equation~\eqref{ab} below. This approach to reduce the singularity was first noted in~\cite{zinchenko-rother-davis-97}, pp.~1503-04, and quoted in~\cite{pozrikidis-01}, p.~284. We multiply the identity~\eqref{sub} by ${f(\textbf{y})}$ and subtract from~\eqref{adjdl}, obtaining the modified adjoint double layer,
\begin{equation}\label{madl}
    v(\bd{y}) = \int_{S}   \frac{ ( f(\textbf{x}) \textbf{n} (\textbf{y}) + f(\textbf{y})\textbf{n} (\textbf{x})) \cdot (\textbf{y} - \textbf{x}) } {4\pi |\textbf{x} - \textbf{y}| ^3} \, dS(\textbf{x}) + \frac{1}{2} f(\textbf{y}).
\end{equation}

\begin{theorem} 
The modified adjoint double layer~\eqref{madl} has a bounded integrand at $\textbf{x}=\textbf{y}$.
\end{theorem}

\begin{proof}
We write in the numerator of~\eqref{madl}, 
\begin{equation}\label{split}
    f(\textbf{x})\textbf{n}(\textbf{y}) + f(\textbf{y})\textbf{n}(\textbf{x}) = f(\textbf{x})(\textbf{n}(\textbf{y}) + \textbf{n}(\textbf{x})) + (f(\textbf{y}) - f(\textbf{x}))\textbf{n}(\textbf{x}).
\end{equation}
The Taylor expansion of the first term on the right hand side of~\eqref{split} amounts to subtracting the numerators of~\eqref{b} and~\eqref{a}, so we get 
\begin{equation}\label{ab}
    f(\textbf{x})(\textbf{n}(\textbf{y}) + \textbf{n}(\textbf{x})) \cdot (\textbf{y} -\textbf{x}) = O(|\alpha |^3 ).
\end{equation}

The dot product of the second term with ${(\textbf{y} -\textbf{x})}$ is the same as the numerator in the double layer integrand in~\eqref{dlsub} and is therefore bounded also. 

\end{proof}

A more detailed analysis of the terms in the expansion of~\eqref{ab} is given in the appendix. 

%%%%%%%%%%%%%%%%%%%%%%%%%%%%%%%%%%%%%%%%%
%%%%%%%%%%%%%%%%%%%%%%%%%%%%%%%%%%%%%%%%%

\section{Numerical Method}
\label{sec: Regularization}

\subsection{Kernel regularization}

To evaluate the modified adjoint double layer~\eqref{madl} numerically, we follow the approach in~\cite{beale-04} and replace the Green's function with a smooth version,
\begin{equation}\label{lamp}
    G_\delta (\textbf{x}- \textbf{y})=G(\textbf{x}- \textbf{y}) \erf(|\textbf{x}- \textbf{y}|/\delta ),
\end{equation}
where $\erf$ is the error function, 
\begin{equation}
    \erf (\rho) = \frac{2}{\sqrt{\pi}} \int _{-\infty} ^{\rho} e^{-t^2} dt,
\end{equation}
and ${\del >0}$ is the regularization parameter. The normal derivative of the Green's function then becomes
\begin{equation}\label{dGdn_del}
    \frac{\partial }{\partial n( \textbf{y})} G_ {\delta}( \textbf{x}- \textbf{y}) = \textbf{n}(\textbf{y}) \cdot \nabla_\bd{y} G_ \delta (\textbf{x}- \textbf{y}) = \textbf{n}(\textbf{y}) \cdot \nabla_\bd{y} G (\textbf{x}- \textbf{y}) \, s(| \textbf{x}- \textbf{y} | / \delta ),
\end{equation}
where $s$ is the shape factor, and with the choice of regularization~\eqref{lamp}, we have 
\begin{equation}\label{s_order3}
    s(\rho) = \erf (\rho) - \frac{2}{\sqrt{\pi}} \rho e ^{-\rho ^2}.
\end{equation}
To compute the adjoint double layer~\eqref{adjdl}, it is first modified as in~\eqref{madl}, and then replaced with its regularized version,
\begin{equation}\label{madl_del}
    v_\del (\textbf{y}) = \int_{S} \frac{ ( f(\textbf{x}) \textbf{n} (\textbf{y}) + f(\textbf{y})\textbf{n} (\textbf{x})) \cdot (\textbf{y} - \textbf{x})\, s(| \textbf{x}- \textbf{y} | / \delta )} {4\pi |\textbf{x} - \textbf{y}| ^3} dS(\textbf{x}) + \frac{1}{2} f(\textbf{y}).
\end{equation}

Below we describe the quadrature used to compute~\eqref{madl_del}. The numerical error consists of two parts, one being the regularization error due to replacing $G$ with $G_\del$, and the other due to the quadrature itself. We first discuss the regularization error and then the quadrature and its error.

\subsection{Analysis of regularization error}
We estimate the error due to replacing $G$ with $G_\delta$, i.e, the replacement of the modified adjoint double layer potential~\eqref{madl} with~\eqref{madl_del} for points on the surface ${\textbf{y}\in S}$. We first show that with the choice of smoothing in~\eqref{lamp}, the regularization error is ${O(\delta ^3 )}$. We then describe a simple way to increase the accuracy to ${O(\delta ^5 )}$ by modifying the choice of $G_\delta$.

We first note that the modified adjoint double layer~\eqref{madl} can be split into two parts using~\eqref{split}, where the second part is the double layer potential with subtraction which was treated in~\cite{beale-04} and shown to be $O(\delta^3)$. Thus we only need to examine the first part of the integral given by
\begin{equation}
    \int _S f(\textbf{x})(\textbf{n}(\textbf{y})+\textbf{n}(\textbf{x}))\cdot \nabla G(\textbf{x}-\textbf{y}) \, dS(\textbf{x}).
\end{equation}
The regularized version of this integral is
\begin{equation}
    \int _S f(\textbf{x})(\textbf{n}(\textbf{y})+\textbf{n}(\textbf{x}))\cdot \nabla G_{\delta}(\textbf{x}-\textbf{y}) \, dS(\textbf{x}).
\end{equation}
Since the regularization error is localized, we can assume the function ${f=0}$ outside of one coordinate patch. We can write the error as an integral on this patch, regarding $f$ as a function of $\alpha$, and assuming ${\textbf{x}(0) = \textbf{y} = \textbf{0}}$,
\begin{equation}\label{error}
    \epsilon = \int [\nabla G_\delta ( \textbf{x} (\alpha ) ) - \nabla G( \textbf{x}( \alpha ) )] \cdot ( \textbf{n}_0 + \textbf{n} (\alpha ) ) f(\alpha ) \, dS(\alpha ).
\end{equation}
Here
\begin{equation}
    \nabla (G_{\delta}-G)=\frac{\textbf{x}}{r} \frac{\partial }{\partial r}(G_{\delta} -G) ,
\end{equation}
\begin{equation}\label{p}
    \frac{\partial }{\partial r}(G_\delta -G)=\frac{1}{4\pi r^2}\phi (r/\delta ),
\end{equation}
where ${r = |\textbf{x}- \textbf{y}|}$, and
\begin{equation}\label{phi}
    \phi (\rho )=-\erfc(\rho )+\rho \erfc'(\rho )=-\erfc(\rho )-(\frac{2}{\sqrt{\pi }})\rho e^{-\rho ^2},
\end{equation}
where $\erfc$ is the complementary error function, ${\erfc(\rho) = 1-\erf (\rho )}$.
Then~\eqref{error} becomes
\begin{equation}
    \epsilon =\frac{1}{4\pi} \int \frac{\textbf{x}(\alpha )\cdot (\textbf{n}(0)+\textbf{n}(\alpha ))}{r^3}\phi (r/\delta )f(\alpha )\, dS(\alpha ).
\end{equation}
Using the expansion in~\eqref{xtaylor2} for $\textbf{x}$ in the $\alpha$-coordinate system, we get
\begin{equation}\label{rsq}
    r^2 =|\textbf{x}|^2=|\alpha |^2  +O(|\alpha |^4 ),
\end{equation}
so that
\begin{equation}
   r/|\alpha| = 1 + O(|\alpha |^2)
\end{equation}
is smooth near $\alpha = 0$.
We make a change of variables
\beq
  \xi = (r/|\alpha|)\alpha
\eeq
so that $|\xi| = r$.  Near zero we have $\alpha = \xi + O(|\xi|^3)$ and
$|\pa\alpha/\pa\xi| = 1 + O(|\xi|^2)$.

We can now write the smoothing error in the form
\begin{equation}\label{error2}
    \epsilon =\frac{1}{4\pi}\int \frac{\phi (|\xi|/\delta )}{|\xi|^{3}} \,w(\xi)\,d\xi,
\end{equation}
where we have 
\begin{equation}\label{bigguy}
   w(\xi)=[ \textbf{x} (\xi) \cdot (\textbf{n}_0 +\textbf{n}(\alpha)) ] \,
         f(\alpha )|\pa\alpha/\pa\xi||\textbf{T}_1\times\textbf{T}_2| .
\end{equation}
With the Taylor expansions~\eqref{xtaylor2} and~\eqref{ntaylor2}, 
as in \eqref{ab}, we have
\begin{equation}
    \textbf{x} (\alpha) \cdot (\textbf{n}_0 +\textbf{n} (\alpha) )= 
          O(|\alpha |^3) = O(|\xi|^3),
\end{equation}
and~\eqref{bigguy} becomes
\begin{equation}\label{w}
     w(\xi)= f_0 O(|\xi|^3) +O(|\xi|^4).
\end{equation}
The third order terms in $\xi$ are odd and
will contribute 0 to the integral~\eqref{error2}. We check that the remainder
$R(\xi) = O(|\xi|^4 )$ is negligible: with the change of variables ${\xi = \delta \zeta}$, we can write $R(\xi) = \delta ^4 \tilde{R}(\zeta )$ for some bounded function $\tilde{R}$ which is
$O(|\zeta|^4)$. The contribution to~\eqref{error2} of the remainder is 
\begin{equation}\label{err_cubic}
    (4\pi)^{-1} \delta ^{-3+4+2}\int \frac{ \phi (|\zeta |) }{ |\zeta |^{3} } \,\tilde{R} (\zeta )\, dS(\zeta ) = O(\delta ^3 ).
\end{equation}
We have now shown that the error $\epsilon$ in \eqref{error} is $O(\delta^3)$, as is the second part of the regularization error.

\subsection{Fifth order regularization}

The error in~\eqref{err_cubic} can be improved to ${O(\delta ^5 )}$ with a simple change in $G_\delta$, by using a modified shape function,
\begin{equation}\label{s_order5}
s_5 (r) = \erf(r) - (2/ \sqrt{\pi} )(r-2r^3 /3 )\, e^{-r^2},
\end{equation}
in place of the shape function $s$ given in~\eqref{s_order3}. The modification to $\nabla G$ is then
\beq
\nabla G _\delta (\textbf{x}) = \nabla G(\textbf{x})\,  s _5 ( |\textbf{x}| / \delta ),
\eeq
see~\cite{beale-04}, p.~607. This shape function is chosen by modifying the previous $s$ so that the integral as in~\eqref{err_cubic} with fourth order terms in $\tilde{R}(\zeta)$ is zero.

\subsection{Quadrature and combined error}
\label{sec: Quad_error}

To discretize the surface integral in~\eqref{madl_del}, we use the direct quadrature method of~\cite{beale-ying-wilson-16, wilson-10} which does not require coordinate charts or triangulations on the surface but instead uses projections on coordinate planes.  We give a brief description.
First choose an angle $\theta$ (we use $70^o$ in our numerical tests) and define a partition of unity on the unit sphere, consisting of functions ${\psi_1, \psi_2, \psi_3}$ with ${\Sigma_i \psi_i \equiv 1}$ such that ${\psi_i(\bd{n}) = 0}$ if ${|\bd{n}\cdot\bd{e}_i| \leq \cos{\theta}}$, where $\bd{e}_i$ is the $i$th coordinate vector. The partition of unity functions $\psi_i$ are constructed from the $C^\infty$ bump function ${b(r) = e^{r^2/(r^2 - 1)}}$ for ${|r| < 1}$ and zero otherwise. For a chosen grid size $h$, a set $R_3$ of quadrature points consists of points $\bf{x}$ on the surface of the form ${(j_1h,j_2h,x_3)}$ such that ${|\bd{n(x}) \cdot \bd{e}_3| \geq \cos{\theta}}$, where ${\bd{n}(\bd{x})}$ is the unit normal at $\bd{x}$. Sets $R_1$ and $R_2$ are defined similarly. For a function $f$ on the surface the integral is computed as
\beq\label{Quadrature}
\int_S f(\bd{x}) \,dS(\bd{x}) \approx \sum_{i=1}^3 \sum_{\bd{x} \in R_i} f(\bd{x}) w(\bd{x}), \qquad  w(\bd{x}) = \frac{\psi_i(\bd{n}(\bd{x}))}{| \bd{n}(\bd{x})\cdot\bd{e}_i |} \,h^2.
\eeq
The quadrature is effectively reduced to the trapezoidal rule without boundary.  Thus for regular integrands the quadrature has arbitrarily high order accuracy, limited only by the degree of smoothness of the integrand and surface. The points in $R_i$ can be found by a line search since they are well separated; see~\cite{beale-ying-wilson-16, wilson-10}. This approach of first regularizing the kernel and applying this straightforward quadrature has been shown effective for single and double layer integrals in the context of Laplace and Stokes equations~\cite{beale-ying-wilson-16, tlupova-beale-19}. 

The full error in computing the adjoint double layer integral $v$, as written in~\eqref{madl}, consists of the regularization error ${v_\del - v}$, where $v_\del$ is defined in~\eqref{madl_del}, and the discretization error from the quadrature of $v_\del$. The first error is $O(h^p)$, with ${p = 3}$ or $5$ here.  As in ~\cite{beale-04,beale-ying-wilson-16} the leading term in the discretization error is of the form ${O(h^2 e^{-c_0(\delta/h)^2})}$.  This error decreases rapidly as ${\delta/h}$ increases, and the regularization error dominates if this ratio is large enough.  We find that ${\delta/h = 3}$ works well in practice with ${p=5}$ and ${\delta/h = 2}$ with ${p=3}$, but for convergence as ${h \to 0}$ we can take ${\delta = ch^q}$ with ${q < 1}$ so that the total error is ${O(h^{pq})}$; see~\cite{beale-ying-wilson-16, tlupova-beale-19}.  Estimates for leading discretization errors in~\cite{beale-20-arxiv} support the expectation that they can be neglected.

\subsection{Discrete integral equation}

We write the integral equation~\eqref{huu} as
\beq\label{IntEq}
    A f = g,
\eeq
where $g$ is the specified value of ${\pa u/ \pa n}$ on the boundary, $f$ is the density we will solve for, ${A = -I/2 + T^*}$, $I$ is the identity and $T^*$ is the integral operator with the kernel of the adjoint double layer potential,
\begin{equation}
    \frac{\partial }{\partial n( \textbf{y})}G( \textbf{x}- \textbf{y}) = \frac{\textbf{n}(\textbf{y})\cdot (\textbf{y}- \textbf{x})}{4\pi |\textbf{x}-\textbf{y}|^3}.
\end{equation}
Assume the domain and its complement are connected. The range of $A$ consists of functions with integral zero on the boundary surface, ${\int_S g dS = 0}$, since they are normal derivatives of harmonic functions in the domain interior~\cite{folland-95}. The solution $u$ of the interior Neumann problem is unique up to an arbitrary constant. We also know that $A$ has a one-dimensional null space, and the null functions do not have integral zero on the surface. 

We will solve a discrete version of the equation~\eqref{IntEq},
\beq\label{IntEq_discrete}
    A_h f_h = g_h,
\eeq
where $f_h$ and $g_h$ have values at the quadrature points. We expect the range of $A_h$ will have codimension one, so we modify the discrete equation as 
\beq\label{IntEq_augmented}
A_h f_h + a_h\textbf{1} = g_h,
\eeq
where $a_h$ is a new scalar unknown and $\textbf{1}$ is a vector of all 1's, i.e. the constant function 1.  This allows the freedom to augment $g_h$ so that ${g_h - a_h\textbf{1}}$ is in the range of $A_h$.  We expect $a_h$ will be very small because $g_h$ is close to being in the range of $A_h$. In numerical experiments, $a_h$ was small and decreased with grid refinement. Since we have added a new unknown, we will also add a new equation.  A natural one is to require the discrete version of ${\int_S f dS = 0}$. This is a straightforward calculation using the already computed weights of the quadrature~\eqref{Quadrature}. 
Our approach is similar to that for the treatment of the Neumann problem by finite differences in~\cite{hackbusch-17}, Sec.~4.7, the Neumann integral equation in~\cite{xie-ying-20},
and the Poisson equation~\cite{beale-20}, Sec.~3.

The discrete equation~\eqref{IntEq_augmented} can be solved by any suitable method, such as a Krylov subspace method (GMRES being the most common), and for larger problems, a fast summation technique such as the fast multipole method~\cite{greengard-oneil-rachh-vico-21} or a treecode~\cite{wang-krasny-tlupova-20, boateng-tlupova-22} could be applied. To limit the scope, we solve the integral equation by successive approximations.  For the exact problem the iteration is
\begin{equation}\label{iteration}
    f^{n+1} = (1-\beta ) f^n + 2 \beta T^* f^n - 2\beta g,
\end{equation}
with $f^{0} = 0$. The iterations~\eqref{iteration} are analogous to those for the Dirichlet problem~\cite{kress-99} used in~\cite{beale-04}; either converges for ${0 < \beta < 1}$. For~\eqref{iteration} the iterates stay within the subspace of functions with integral zero.

We write the discrete version of~\eqref{iteration} as
\beq\label{iteration_discrete}
    f_h^{n+1} + a_h^{n+1}\textbf{1} =
(1-\beta ) f_h^n + 2 \beta T_h^* f_h^n - 2\beta g_h, \qquad
\sum_h f^{n+1}_h  w_h = 0,
\eeq
where $w_h$ are the quadrature weights, $a_h^{n+1}$ is a scalar, 
and $T_h^*$ is the matrix corresponding with the discretized operator $T^*$, computed with the singularity modification and the fifth order kernel.  To find $f_h^{n+1}$,
we first compute the right side $f^*$ of the first equation, next define
$a_h^{n+1}$ as the weighted average of $f^*$, and then set
$f_h^{n+1} = f^* - a_h^{n+1}\textbf{1}$, so that the second equation holds.
Once we have the density $f_h$, we can find $u_h$ using the single layer representation~\eqref{sl}, which will differ from the exact solution by an arbitrary constant. To compare $u_h$ to the exact solution $u$, we adjust the constant and set our computed solution to 
\beq\label{adjust_constant}
    u_h(\textbf{x}) = u_h(\textbf{x}) -u_h(\textbf{0}) + u(\textbf{0}).
\eeq
The integral in~\eqref{sl} is computed as in \cite{beale-ying-wilson-16}.

%%%%%%%%%%%%%%%%%%%%%%%%%%%%%%%%%%%%%%%%%
%%%%%%%%%%%%%%%%%%%%%%%%%%%%%%%%%%%%%%%%%

\section{Numerical results}
\label{sec: Numerical Results}

We present several numerical examples to illustrate our method. For the first test, we use a known solution $f$ based on a spherical harmonic to evaluate the modified adjoint double layer $v$ directly, without regularization~\eqref{madl}, then with regularization~\eqref{madl_del}, using the third order kernel~\eqref{s_order3} and the fifth order kernel~\eqref{s_order5}. For the second test, we use the same exact solution to set the Neumann condition ${\pa u/ \pa n = g_{exact}}$ and solve the boundary value problem~\eqref{naminn} for $u$. To do this, we first solve the integral equation~\eqref{huu} in its discrete form~\eqref{iteration_discrete} for $f$. We then use this solution to compute $u$ at the quadrature points as the single layer potential~\eqref{sl} and adjust the constant using~\eqref{adjust_constant}. We apply the fifth order regularization for the single layer potential given in~\cite{beale-ying-wilson-16}. We note that the solution $u$ can also be computed at points off the surface, as needed in some applications, with high accuracy using the method of corrections in~\cite{beale-ying-wilson-16}.  For the final test, we use a different exact solution for which the solution of the integral equation is not known, and solve the Neumann problem on two surfaces, an ellipsoid and a four-atom molecular surface. To see convergence as ${h \to 0}$ we choose ${\delta = ch^q}$ with $q = 4/5$ or $2/3$.  With $q = 4/5$ we observe the overall error of $O(h^4)$ as expected, and with $q = 2/3$ we find accuracy slightly higher than the predicted order.

In our first example, the surface is the unit sphere ${x_1^2+x_2^2+x_3^2 = 1}$, and both the solution $u$ of the boundary value problem~\eqref{naminn} and the solution $f$ of the integral equation~\eqref{huu} are constructed based on a spherical harmonic. As in~\cite{beale-04}, we first define 
\begin{equation}\label{test1_f}
    f(\textbf{x} ) = 1.75(y_1 - 2y_2 )(7.5y_3 ^2 - 1.5), \quad \textbf{y}=M\textbf{x},
\end{equation}
where ${M = (1/ \sqrt{6}) \left(\sqrt{2}(1,1,1)^T , \sqrt{3}(0,1,-1)^T , (-2,1,1)^T \right)}$ is an orthogonal matrix used to avoid rectangular symmetry in the test problem. The functions ${r^3 f(\textbf{x}/r)}$ and ${r^{-4} f(\textbf{x}/r)}$, ${r = |\textbf{x}|}$, are both harmonic. Then, the single layer $u$ due to $f$ is set to these harmonic functions with coefficients adjusted due to the fact that $u$ is continuous and ${\partial u/\partial n}$ has a jump equal to $f$ (see e.g.,~\cite{folland-95}) across the boundary ${r=1}$. We set ${g (\textbf{x}) = (-3/7) f(\textbf{x})}$ for the Neumann condition on the boundary, and 
\beq\label{test1_u}
 u(\textbf{x}) = - r^3 f(\textbf{x}/r) / 7,\quad r<1; \qquad u(\textbf{x}) = - r^{-4} f(\textbf{x}/r) / 7,\quad r>1.
\eeq

We use the known density function $f$ to first test the direct evaluation of the adjoint double layer potential. We compute the left side of equation~\eqref{huu} with exact values $f_{exact}$ at the quadrature points, and call these values $g_{comp}$. We then compare $g_{comp}$ with the exact values ${g_{exact}=(-3/7)f_{exact}}$. We define the error at each quadrature point $\textbf{y}$,
\begin{equation}
e_{h}(\textbf{y}) = g_{comp}(\textbf{y}) - g_{exact}(\textbf{y}),
\end{equation}
and its norms,
\begin{equation}
    \lVert e_h\rVert_2 = \left( \sum_\textbf{y} |e_{h}(\textbf{y})| ^2 / N  \right)^{1/2}, \quad \quad \lVert e_h\rVert_\infty = \max_\textbf{y}(|e_{h}(\textbf{y})| ),
\end{equation}
where $h$ is the grid size in the coordinate planes and $N$ is the number of quadrature points generated. Table~\ref{table:ADL_eval_order3} shows the results for the adjoint double layer that uses the singularity reduction but no smoothing~\eqref{madl} (with the point ${\bd{x}=\bd{y}}$ omitted), and the adjoint double layer with the singularity reduction and third order smoothing~\eqref{madl_del},~\eqref{s_order3}, with ${\del=2h}$.  In Table~\ref{table:ADL_eval_order5} we present the results using the fifth order smoothing~\eqref{s_order5} with ${\del=3h}$ and ${\del=1.5h^{4/5}}$. In each table, the ``Order" column shows the $\log_2$ of the ratio of the $L_2$ error of that row to the row above it, that is, ${\log_2\lVert e_{2h}\rVert_2 / \lVert e_h\rVert_2}$. 

The adjoint double layer with the singularity modification but no smoothing has an error of about $O(h^2)$, and the maximum errors are about 7 times as large as the $L^2$ errors. When regularization is used, the maximum errors are about twice as large as the $L^2$ errors. For the third order kernel we take ${\del = 2h}$. For a higher order kernel we can use a larger smoothing parameter, such as ${\del=3h}$. This point was discussed in further detail in~\cite{tlupova-beale-19}, sec.~4.2. Both third and fifth order kernels give smaller errors than the kernel without regularization when $h$ is sufficiently small, and the errors with the fifth order kernel are much smaller than the third order kernel. We observe the expected order of convergence as the grid is refined, $O(h^3)$ for the third order kernel and $O(h^{5q})$ for the fifth order kernel with ${\del=c h^{q}}$, where ${q<1}$. In practice, ${\del=3h}$ works well also, however for convergence for ${h\to 0}$ we take ${\del=1.5h^{4/5}}$ or a similar choice, as discussed in sec.~\ref{sec: Quad_error}. Overall, the error in the adjoint double layer significantly improves when using the singularity modification~\eqref{madl} in conjunction with the fifth order kernel~\eqref{s_order5}.

% UNIT SPHERE DIRECT EVAL NO REG + 3RD ORDER
\begin{table}[!htb]
\centering
\begin{tabular}{|c|c||c|c|c||c|c|c|}
\hline
\multirow{2}{*}{$h$} & \multirow{2}{*}{$N$}  &
\multicolumn{3}{c||}{\ \ \ no regularization\ \ \ } &
\multicolumn{3}{c|}{\ \ \ 3rd order $\del=2h$\ \ \ } \\
 & & \ \ \ $\lVert e_h\rVert_\infty$ \ \ \ & \ \ \ $\lVert e_h\rVert_2$ \ \ \ & 
Order & \ \ \ $\lVert e_h\rVert_\infty$ \ \ \ & \ \ \ $\lVert e_h\rVert_2$ \ \ \ & 
Order \\ 
\hline
1/16 & 4302 & 3.56e-03 & 5.47e-04  &  & 8.90e-03 & 3.99e-03 &  \\
\hline
1/32 & 17070 & 9.25e-04 & 1.46e-04 & 1.9  & 1.11e-03 & 5.00e-04 & 3.0 \\
\hline
1/64 & 68166 & 2.55e-04 & 3.45e-05  & 2.1 & 1.39e-04 & 6.24e-05 & 3.0 \\
\hline
1/128 & 272718 & 6.57e-05 & 8.88e-06 & 2.0 & 1.74e-05 & 7.80e-06 & 3.0 \\
\hline
\end{tabular}
\caption{Unit sphere: direct evaluation of the modified adjoint double layer potential without regularization (left) and using the third order kernel (right) with $\del=2h$.}
\label{table:ADL_eval_order3}
\end{table}

% UNIT SPHERE DIRECT EVAL 5TH ORDER
\begin{table}[!htb]
\centering
\begin{tabular}{|c||c|c|c||c|c|c|}
\hline
\multirow{2}{*}{$h$} &
\multicolumn{3}{c||}{\ \ \ 5th order $\del=3h$\ \ \ } &
\multicolumn{3}{c|}{\ \ \ 5th order $\del=1.5h^{4/5}$\ \ \ } \\
 & \ \ \ $\lVert e_h\rVert_\infty$ \ \ \ & \ \ \  $\lVert e_h\rVert_2$ \ \ \ & 
Order & \ \ \ $\lVert e_h\rVert_\infty$ \ \ \ & \ \ \ $\lVert e_h\rVert_2$ \ \ \ & 
Order \\
\hline
1/16 & 9.27e-04 & 4.10e-04 & & 5.55e-04 & 2.40e-04  &   \\
\hline
1/32 & 3.08e-05 & 1.38e-05 & 4.9 & 3.08e-05 & 1.38e-05  & 4.1 \\
\hline
1/64 & 7.64e-07 & 3.40e-07 & 5.3 & 1.53e-06 & 6.86e-07  & 4.3 \\
\hline
1/128 & 2.39e-08 & 1.07e-08 & 5.0 & 9.64e-08 & 4.33e-08 & 4.0 \\
\hline
\end{tabular}
\caption{Unit sphere: direct evaluation of the modified adjoint double layer potential using the fifth order kernel with $\del=3h$ (left) and $\del=1.5h^{4/5}$ (right).}
\label{table:ADL_eval_order5}
\end{table}

For our second test, we still use the exact solution~\eqref{test1_f},~\eqref{test1_u} and solve the boundary value problem~\eqref{naminn} for $u$. We set the Neumann condition ${\pa u/ \pa n = g_{exact}}$ and solve the integral equation~\eqref{huu} in its discrete form~\eqref{iteration_discrete} for $f$, then compute $u$ at the quadrature points as the single layer potential~\eqref{sl}, adjusting the constant as in~\eqref{adjust_constant}. We apply the fifth order regularization for the single layer potential given in~\cite{beale-ying-wilson-16}. We choose $\beta = 0.7$ in iterations~\eqref{iteration_discrete}, which are performed until ${\lVert f^{n+1} - f^n\rVert_\infty < 10^{-8}}$. Table~\ref{table:sphere_err_in_u} shows the errors in $f$ and $u$, using ${\del=3h}$. ``Error in $f$" is the error between the final iterate $f^n$ and the exact values of $f$ given by~\eqref{test1_f}. ``Error in $u$" is the error between the single layer potential computed at the quadrature points using $f^n$, and the exact values of $u$ given by~\eqref{test1_u}. We observe the expected fifth order convergence in both $f$ and $u$, and the error values are similar in magnitude to those in Table~\ref{table:ADL_eval_order5}.

% UNIT SPHERE Integral Equation: 23 iterations to tolerance 10^(-8)
\begin{table}[!htb]
\centering
\begin{tabular}{|c||c|c|c||c|c|c|}
\hline
%$\delta = 3h$
\multirow{2}{*}{$h$} & 
\multicolumn{3}{c||}{\ \ \ Error in $f$ \ \ \ } &
\multicolumn{3}{c|}{\ \ \ Error in $u$ \ \ \ } \\
 & \ \ \ $\lVert e_h\rVert_\infty$ \ \ \ & \ \ \  $\lVert e_h\rVert_2$ \ \ \ & 
Order & \ \ \ $\lVert e_h\rVert_\infty$ \ \ \ & \ \ \ $\lVert e_h\rVert_2$ \ \ \ & 
Order \\
\hline
1/16 & 2.18e-03 & 9.57e-04  &  & 9.51e-04 & 3.62e-04 &  \\
\hline
1/32 & 7.19e-05 & 3.23e-05  & 4.9 & 7.43e-05 & 1.14e-05 & 5.0 \\
\hline
1/64 & 1.78e-06 & 7.90e-07  & 5.4 & 2.04e-06 & 3.58e-07 & 5.0 \\
\hline
1/128 & 5.02e-08 & 2.25e-08 & 5.1 & 4.88e-08 & 1.13e-08 & 5.0 \\
\hline
\end{tabular}
\caption{Unit sphere: errors in the solution $f$ of the integral equation and the solution $u$ of the Neumann problem with $\del=3h$.}
\label{table:sphere_err_in_u}
\end{table}

For our third test, the surface is the ellipsoid ${x_1 ^2 + 4x_2 ^2 + 4x_3 ^2 = 1}$, and we use the harmonic function
\begin{equation}\label{test2_u}
    u(\textbf{x} )=\text{exp}(y_1 + 2y_2 )\text{cos}(\sqrt{5} y_3 ),\quad \textbf{y} = M \textbf{x},
\end{equation}
where $M$ is the same matrix as in~\eqref{test1_f}. We prescribe ${\pa u/\pa n = g}$ on the surface $S$, where $g$ is computed analytically. In this case $f$ is not known and the integral equation has to be solved. We perform the modified iterations~\eqref{iteration_discrete}, again with ${\beta=0.7}$, and until a tolerance $10^{-8}$ is reached between two consecutive iterates. We then compute the solution $u$ at the quadrature points and redefine it to adjust the constant as in equation~\eqref{adjust_constant}. Table~\ref{table:ellipsoid} shows the errors in $u$  and again the order of convergence as ${\log_2 \lVert e_{2h}\rVert_2 / \lVert e_h\rVert_2}$. To expect convergence as ${h\to 0}$, we generally set ${\del=c h^q}$ with ${q<1}$, and we tested two values of the regularization parameter, ${\del = 0.75 h^{2/3}}$ and ${\del=1.5 h^{4/5}}$ as earlier. The expected order of convergence is then ${5q}$, or 10/3 and 4 respectively, which can be seen in Table~\ref{table:ellipsoid}.

% ELLIPSOID del=3h
%\begin{table}[!htb]
%\centering
%\begin{tabular}{|c|c||c|c|c|}
%\hline
%\multirow{2}{*}{$h$} & \multirow{2}{*}{$N$} &
%\multicolumn{3}{c|}{\ \ \ Error in $u$ \ \ \ }  \\
% & & \ \ \ $\lVert e_h\rVert_\infty$ \ \ \ & \ \ \  $\lVert e_h\rVert_2$ \ \ \ & Order \\
%\hline
%1/16 & 1766 & 0.0710 & 0.0087 & \\
%\hline
%1/32 & 6958 & 0.0067 & 5.11e-04  & 4.09 \\
%\hline
%1/64 & 27934 & 3.16e-04 & 2.41e-05  & 4.41 \\
%\hline
%1/128 & 112006 & 1.30e-05 & 6.18e-07 & 5.29 \\
%\hline
%\end{tabular}
%\caption{Ellipsoid: errors in the solution $u$ of the Neumann problem with $\del = 3h$.}
%\label{table:ellipsoid_err_in_u_del=3h}
%\end{table}

% ELLIPSOID:  70 iterations to tolerance 10^(-8)
\begin{table}[!htb]
\centering
\begin{tabular}{|c|c||c|c|c||c|c|c|}
\hline
\multirow{2}{*}{$h$} & \multirow{2}{*}{$N$} &
\multicolumn{3}{c||}{\ \ \ Error in $u$ with $\del = 0.75 h^{2/3}$ \ \ \ } &  \multicolumn{3}{c|}{\ \ \ Error in $u$ with $\del = 1.5 h^{4/5}$ \ \ \ } \\
 &  & \ \ \ $\lVert e_h\rVert_\infty$ \ \ \ & \ \ \  $\lVert e_h\rVert_2$ \ \ \ & 
Order & \ \ \ $\lVert e_h\rVert_\infty$ \ \ \ & \ \ \  $\lVert e_h\rVert_2$ \ \ \ & 
Order \\
\hline
1/16 & 1766 & 3.90e-02 & 3.93e-03 & & 4.67e-02 & 5.69e-03  & \\
\hline
1/32 & 6958 & 3.24e-03 & 3.34e-04 & 3.56 & 6.66e-03 & 5.11e-04 & 3.48 \\
\hline
1/64 & 27934 & 3.16e-04 & 2.41e-05 & 3.79 & 5.25e-04 & 3.54e-05 & 3.85 \\
\hline
1/128 & 112006 & 2.67e-05 & 1.57e-06 & 3.94 & 3.37e-05 & 1.96e-06 & 4.17 \\
\hline
\end{tabular}
\caption{Ellipsoid: errors in the solution $u$ of the Neumann problem.}
\label{table:ellipsoid}
\end{table}

In our final test, the surface is a four-atom molecule~\cite{beale-ying-wilson-16}, given by ${\sum_{k=1}^4 \exp(- |{\bf x} - {\bf x}_k|^2/r^2) = c}$, with centers ${(\sqrt{3}/3,0,-\sqrt{6}/12)}$, ${(-\sqrt{3}/6,\pm .5,-\sqrt{6}/12)}$, ${(0,0,\sqrt{6}/4)}$ and ${r = 0.5}$, ${c = 0.6}$. The solution is taken as the harmonic function in~\eqref{test2_u}, and we repeat the above tests. The errors shown in Table~\ref{table:molecule} are similar to those for the ellipsoid, with the expected $O(h^4)$ again observed when using ${\del = ch^{4/5}}$.

% MOLECULE:  70 iterations to tolerance 10^(-8)
\begin{table}[!htb]
\centering
\begin{tabular}{|c|c||c|c|c||c|c|c|}
\hline
\multirow{2}{*}{$h$} & \multirow{2}{*}{$N$} &
\multicolumn{3}{c||}{\ \ \ Error in $u$ with $\del = 0.75 h^{2/3}$ \ \ \ } &  \multicolumn{3}{c|}{\ \ \ Error in $u$ with $\del = 1.5 h^{4/5}$ \ \ \ } \\
 &  & \ \ \ $\lVert e_h\rVert_\infty$ \ \ \ & \ \ \  $\lVert e_h\rVert_2$ \ \ \ & 
Order & \ \ \ $\lVert e_h\rVert_\infty$ \ \ \ & \ \ \  $\lVert e_h\rVert_2$ \ \ \ & 
Order \\
\hline
1/16 & 2392 & 3.20e-02 & 3.50e-03 & & 2.40e-02 & 4.39e-03  & \\
\hline
1/32 & 9562 & 3.13e-03 & 2.84e-04 & 3.62 & 2.66e-03 & 3.64e-04 & 3.59 \\
\hline
1/64 & 38354 & 2.59e-04 & 1.44e-05 & 4.30 & 1.91e-04 & 2.19e-05 & 4.05 \\
\hline
1/128 & 153399 & 7.83e-06 & 1.04e-06 & 3.79 & 8.80e-06 & 1.30e-06 & 4.07 \\
\hline
\end{tabular}
\caption{Molecular surface: errors in the solution $u$ of the Neumann problem.}
\label{table:molecule}
\end{table}

%%%%%%%%%%%%%%%%%%%%%%%%%%%%%%%%%%%%%%%%%
%%%%%%%%%%%%%%%%%%%%%%%%%%%%%%%%%%%%%%%%%
%\clearpage

\section{Appendix}
To take a closer look at the principal terms in the expansion of the modified adjoint double layer~\eqref{madl}, we consider additional terms in~\eqref{xtaylor2},~\eqref{ntaylor2}, and~\eqref{ftaylor},
\begin{eqnarray}\label{xte}
     \textbf{x}(\alpha )= \textbf{T}_1 (0)\alpha_1 &+&\textbf{T}_2 (0)\alpha_2 +\frac{1}{2} \kappa _1 \textbf{n}_0 \alpha _1 ^2 +\frac{1}{2} \kappa _2 \textbf{n}_0 \alpha _2 ^2 + \frac{1}{6}\textbf{x}_{111}\alpha _{1}^3 + \frac{1}{2}\textbf{x}_{112}\alpha _{1}^2 \alpha _{2} \nonumber \\ &+& \frac{1}{2}\textbf{x}_{122}\alpha _{1} \alpha _{2}^2 + \frac{1}{6}\textbf{x}_{222}\alpha _{2}^3 + O(|\alpha |^4 ),
\end{eqnarray}
\begin{equation}\label{nte}
    \textbf{n}(\alpha )=\textbf{n}_0-\kappa _1 \textbf{T}_1 (0)\alpha _1 -\kappa _2 \textbf{T}_2 (0)\alpha _2 + \frac{1}{2}\textbf{n}_{11}\alpha _{11}^2 + \textbf{n}_{12}\alpha _{1}\alpha _{2} + \frac{1}{2}\textbf{n}_{22}\alpha _{22}^2 +O(|\alpha |^3 ),
\end{equation}
\begin{equation}\label{fte}
     f(\alpha )=f_0 +f_1 \alpha _1 +f_2 \alpha _2 + \frac{1}{2}f_{11} \alpha _1 ^2 + f_{12}\alpha _1 \alpha _2 +\frac{1}{2}f_{22} \alpha _2 ^2 +O(|\alpha |^3).
\end{equation}
Then, we have for the numerator in~\eqref{madl},
\begin{eqnarray}\label{heck}
    (f(\alpha )\textbf{n}_0 +f_0 \textbf{n}(\alpha ))\cdot \textbf{x}( \alpha )&=&
      \frac{1}{2}(\kappa _1 \alpha _1 ^2 +\kappa _2 \alpha _2 ^2 )(f_1 \alpha _1 +f_2 \alpha _2 ) \nonumber \\
    &+& \frac{1}{2} \textbf{n}_{11}\cdot \textbf{T}_1 f_0 \alpha _1 ^3 + \textbf{n}_{12}\cdot \textbf{T}_1 f_0 \alpha _1 ^2 \alpha _2 + \frac{1}{2}\textbf{n}_{22}\cdot \textbf{T}_1 f_0 \alpha _1 \alpha _2 ^2 \nonumber \\
    &+&  \frac{1}{2}\textbf{n}_{11}\cdot \textbf{T}_2 f_0 \alpha _1 ^2 \alpha _2 + \textbf{n}_{12}\cdot \textbf{T}_2 f_0 \alpha _1 \alpha _2 ^2+ \frac{1}{2}\textbf{n}_{22}\cdot \textbf{T}_2 f_0 \alpha _2 ^3  \nonumber\\
    &+&  \frac{1}{3} \textbf{n}_0 \cdot \textbf{x}_{111} f_0 \alpha _1 ^3 +\textbf{n}_0 \cdot \textbf{x}_{112} f_0 \alpha _1 ^2 \alpha _2  \nonumber\\
    &+& \textbf{n}_0 \cdot \textbf{x}_{122} f_0 \alpha _1 \alpha _2 ^2 +\frac{1}{3} \textbf{n}_0 \cdot \textbf{x}_{222} f_0 \alpha _2 ^3 + O(|\alpha |^4 ). 
\end{eqnarray}
To simplify~\eqref{heck}, we investigate the terms such as ${\textbf{n}_{11}\cdot \textbf{T} _1}$, ${\textbf{n}_{12}\cdot \textbf{T} _1}$, etc. We start with the definition of the normal vector stated in~\eqref{2nn}. Differentiating~\eqref{2nn} twice in $\alpha _1$, we get
\begin{equation}\label{n11}
    \textbf{n}_{11} =\frac{\partial ^2}{\partial \alpha _1 ^2}(\textbf{T}_1 \times \textbf{T}_2 ) \frac{1}{\sqrt{g}} +2 \frac{\partial}{\partial \alpha _1} (\textbf{T}_1 \times \textbf{T}_2 ) \cdot \frac{\partial}{\partial \alpha _1} \left( \frac{1}{\sqrt{g}} \right) + (\textbf{T}_1 \times \textbf{T}_2 ) \frac{\partial ^2}{\partial \alpha _1 ^2} \left( \frac{1}{\sqrt{g}} \right) .
\end{equation}
Evaluating~\eqref{n11} at ${\alpha = 0}$, the second term is equal to zero since Christoffel symbols are zero and the third term will vanish when taking the dot product with $\textbf{T}_1$ and  $\textbf{T}_2$ since it is perpendicular to them. We are left with only the first term which we expand,
\begin{equation}\label{n11_term1}
    \frac{\partial ^2}{\partial \alpha _1 ^2} (\textbf{T}_1 \times \textbf{T}_2 ) = \frac{\partial ^2 \textbf{T}_1 }{\partial \alpha _1 ^2} \times \textbf{T}_2 + 2  \frac{\partial \textbf{T}_1 }{\partial \alpha _1} \times \frac{\partial \textbf{T}_2 }{\partial \alpha _1} + \textbf{T}_1 \times \frac{\partial ^2 \textbf{T}_2 }{\partial \alpha _1 ^2}.
\end{equation}
The second term in~\eqref{n11_term1} is zero at ${\alpha =0}$, and we compute the dot products of the remaining terms with $\textbf{T}_1$ as needed in \eqref{heck},
\begin{equation}
    \left( \frac{\partial ^2 \textbf{T}_1 }{\partial \alpha _1 ^2} \times \textbf{T}_2 \right) \cdot \textbf{T}_1 = \frac{\partial ^2 \textbf{T}_1 }{\partial \alpha _1 ^2} \cdot (\textbf{T}_2 \times \textbf{T}_1) = -\textbf{x}_{111}\cdot \textbf{n}_0
\end{equation}
and 
\begin{equation}
    \left( \textbf{T}_1 \times \frac{\partial ^2 \textbf{T}_2 }{\partial \alpha _1 ^2} \right) \cdot \textbf{T}_1 = \frac{\partial ^2 \textbf{T}_2 }{\partial \alpha _1 ^2} \cdot (\textbf{T}_1 \times  \textbf{T}_1) = 0.
\end{equation}
Finally we obtain at ${\alpha =0}$, $\textbf{n}_{11} \cdot \textbf{T}_1 =-\textbf{x}_{111}\cdot \textbf{n}_0$. By this method we find that generally $\textbf{n}_{ij} \cdot \textbf{T}_k =-\textbf{x}_{ijk}\cdot \textbf{n}_0,\hspace{.2cm} i,j,k=1,2$. Using this, we combine terms in~\eqref{heck} and find the third order terms in~\eqref{ab},
\begin{eqnarray}
    (f(\alpha )\textbf{n}_0 +f_0 \textbf{n}(\alpha ))\cdot \textbf{x}( \alpha )&=&
     \frac{1}{2}(\kappa _1 \alpha _1 ^2 +\kappa _2 \alpha _2 ^2 )(f_1 \alpha _1 +f_2 \alpha _2 ) \nonumber\\
     &-& \frac{1}{6} \textbf{x}_{111} \cdot \textbf{n}_0 f_0 \alpha _1 ^3 - \frac{1}{2} \textbf{x}_{112} \cdot \textbf{n}_0 f_0 \alpha _1 ^2 \alpha _2 \nonumber \\ &-& \frac{1}{2} \textbf{x}_{122} \cdot \textbf{n}_0 f_0 \alpha _1 \alpha _2 ^2 - \frac{1}{6} \textbf{x}_{222} \cdot \textbf{n}_0 f_0 \alpha _2 ^3 +O(|\alpha |^4 ).
\end{eqnarray}

%%%%%%%%%%%%%%%%%%%%%%%%%%%%%%%%%%%%%%%%%

\section*{Acknowledgments}

The work of MS and ST was supported in part by the National Science Foundation grant DMS-2012371.

%%%%%%%%%%%%%%%%%%%%%%%%%%%%%%%%%%%%%%%%%

\section*{Declarations}

\subsection*{Conflict of interest}

The authors declare no competing interests.

%%%%%%%%%%%%%%%%%%%%%%%%%%%%%%%%%%%%%%%%%

\bibliographystyle{plain}
\bibliography{Bib}

\end{document}